\documentclass[12pt,a4paper]{amsart}

\usepackage{amsmath}
\usepackage[margin=3cm]{geometry}
\newcommand*\patchAmsMathEnvironmentForLineno[1]{%
  \expandafter\let\csname old#1\expandafter\endcsname\csname #1\endcsname
  \expandafter\let\csname oldend#1\expandafter\endcsname\csname end#1\endcsname
  \renewenvironment{#1}%
     {\linenomath\csname old#1\endcsname}%
     {\csname oldend#1\endcsname\endlinenomath}}%
\newcommand*\patchBothAmsMathEnvironmentsForLineno[1]{%
  \patchAmsMathEnvironmentForLineno{#1}%
  \patchAmsMathEnvironmentForLineno{#1*}}%
\AtBeginDocument{%
\patchBothAmsMathEnvironmentsForLineno{equation}%
\patchBothAmsMathEnvironmentsForLineno{align}%
\patchBothAmsMathEnvironmentsForLineno{flalign}%
\patchBothAmsMathEnvironmentsForLineno{alignat}%
\patchBothAmsMathEnvironmentsForLineno{gather}%
\patchBothAmsMathEnvironmentsForLineno{multline}%
}
\usepackage[utf8]{inputenc}
\usepackage[T1]{fontenc}
\usepackage[english,ngerman]{babel}
\usepackage{stmaryrd}
\usepackage{dsfont}
\usepackage[matrix, arrow, curve]{xy}
\usepackage{verbatim}
\usepackage{hyperref}
\usepackage[mathlines]{lineno}
\usepackage{enumerate}
\usepackage{tikz-cd}

\usepackage{tikz}
\usetikzlibrary{decorations.pathmorphing, decorations.markings, arrows, matrix, fit, positioning}
\tikzstyle directed=[postaction={decorate,decoration={markings, mark=at position .55 with {\arrow{stealth}}}}]	

\usepackage{color}
\usepackage[noadjust]{cite}

\usepackage{tikz}
\usetikzlibrary{matrix,arrows}

\newcommand*\C{\mathds C}

\newcommand*\Z{\mathds Z}
\newcommand*\N{\mathds N}
\renewcommand*\H{\mathds H}

\newcommand*\R{\mathds R}

\newcommand*{\ord}{\text{ord }}
\newcommand*\Q{\mathds Q}

\renewcommand*\mod{\text{ mod }}
\newcommand*{\rad}{\text{rad}}

\newcommand{\Legendre}[2]{\left( \frac{#1}{#2} \right) }

\DeclareMathOperator{\num}{Num}
\DeclareMathOperator{\den}{Denom}

\DeclareMathOperator{\Gal}{Gal}

\DeclareMathOperator{\Aut}{Aut}

\newtheorem{theorem}{Theorem}[section]
\newtheorem{lemma}[theorem]{Lemma}
\newtheorem{proposition}[theorem]{Proposition}
\newtheorem{example}[theorem]{Example}
\newtheorem*{conjecture}{Conjecture}
\newtheorem{corollary}[theorem]{Corollary}

\theoremstyle{definition}
\newtheorem{definition}[theorem]{Definition}
\newtheorem*{remark}{Remark}

\usepackage{chemarr}
\usepackage{amssymb}

\title{Explicit small heights in infinite non-abelian extensions}

\author[L. Frey]{Linda Frey}
\address{Mathematical Institute,
}
\email{Linda.Frey@unibas.ch}

\usepackage{longtable}\usepackage{fancyhdr}
\begin{document}
\pagestyle{fancy}
\fancyhf{}
\rhead{\thepage \qquad Linda Frey}
\lhead{Explicit Small Heights in Infinite Non-Abelian Extensions}

\selectlanguage{english}
\begin{abstract}
Let $E$ be an elliptic curve over the rational numbers. We consider the infinite extension $\Q(E_{\text{tor}})$ of the rational numbers where we adjoin all coordinates of torsion points of $E$. In this paper we prove an explicit lower bound for the height of non-zero elements in the infinite extension $\Q(E_{\text{tor}})/\Q$ obtained by adjoining the coordinates of all torsion points of $E$ with respect to a Weierstrass model over $\Q$ that are not a root of unity, only depending on the conductor of the elliptic curve. As a side result we give an explicit upper bound for a supersingular prime for an elliptic curve.
\end{abstract}
\maketitle
\section{Introduction}
\label{Introduction}

Kronecker's Theorem states that an algebraic number has absolute logarithmic Weil height zero if and only if it is either zero or a root of unity. A natural question to ask is whether we can find an explicit constant $C>0$ such that the height of any algebraic number is zero or greater or equal to $C$. The fact that the height of $2^{1/n}$ is $\frac{\log 2}{n}$ shows that the answer is no. If we replace the field of algebraic numbers with a smaller field, there is hope that this is true. We say a field has the Bogomolov property if there is a positive constant $C$ such that the height of any non-torsion and non-zero element is greater than $C$. This property was introduced by Bombieri and Zannier in \cite{BZ:infinitebogo}.\\

By Northcott's Theorem every number field satisfies the Bogomolov property. In 1973 Schinzel \cite{MR0360515} proved that $\Q^{\text{tr}}$, the maximal totally real extension of the rational numbers, also satisfies the Bogomolov property. Twenty-seven years later, Amoroso and Dvornicich \cite{MR1740514} proved that $\Q^{\text{ab}}$, the maximal abelian extension of the rational numbers, satisfies the Bogomolov property and they even found an explicit lower bound, namely $\frac{\log 5}{12}$. This bound is almost sharp (there is an element of height $\frac{\log 7}{12}$). The field $\Q^{\text{ab}}$ can be obtained by adjoining $\mu_\infty$, the set of all roots of unity, to the rational numbers. In 2010, Amoroso and Zannier (\cite{MR1817715} effective and \cite{MR2651944} uniform and explicit) in a similar setting proved the following: Let $\alpha\in\overline{\Q}^\times \setminus \mu_\infty$ such that there exists a number field $K$ of degree $d$ over $\Q$ with $K(\alpha)/K$ abelian. Then $h(\alpha) \geq 3^{-d^2-2d-6}$.\\

Now we turn to elliptic curves and create the elliptic curve analogue to $\Q^{\text{ab}}$. Let $E$ be an elliptic curve defined over $\Q$ and let $\Q(E_\text{tor})$ be the smallest field extension of $\Q$ that contains all coordinates of torsion points of $E$. In 2013 Habegger \cite{MR3090783} showed that $\Q(E_\text{tor})$ satisfies the Bogomolov property. The aim of this paper is making this result explicit (not only effective!). Whenever an elliptic curve admits some endomorphism over $\overline{\Q}$ that is not multiplication by an integer, we say the elliptic curve has complex multiplication or short is CM. In the CM case, $\Q (E_{\text{tor}})$ has the Bogomolov property by the result of Amoroso and Zannier \cite{MR1817715} and this becomes explicit using their later work \cite{MR2651944}. So we can concentrate on the other case: For a non-CM elliptic curve, this extension is highly non-abelian and none of the above results can be applied.\\

We now state the two main theorems of this paper.

\begin{theorem}
\label{main}
Let $E$ be an elliptic curve defined over $\Q$ and let $p \geq 5$ be a supersingular prime of $E$ such that the Galois representation $\Gal(\bar{\Q}/\Q) \to \Aut E[p]$ is surjective. Then for all $\alpha \in \Q(E_\text{tor})^\times \setminus \mu_\infty$ we have
\begin{align*}h(\alpha) \geq \frac{(\log p)^5}{10^{31} p^{44}}.\end{align*}
\end{theorem}

After bounding $p$, we get the following theorem where $\vartheta (n) = \sum_{p \text{ prime} \leq n} \log p$.

\begin{theorem}
Let $E$ be an elliptic curve defined over $\Q$ of conductor $N$. Let $\alpha \in \Q(E_\text{tor})^\times \setminus \mu_\infty$. Then with $n = 10^7 \max\{985, \frac1{12}(18 N \log N)+3\}^2$ we have
\begin{align*}
h(\alpha) \geq ((8 N e^{\vartheta(n)})^{N e^{\vartheta(n)} (\log (8 N e^{\vartheta(n)}))^5}  18 N \log N)^{-44}.
\end{align*}
\end{theorem}

For proving these theorems, we dive into Habegger's paper \cite{MR3090783} which states that $\Q(E_\text{tor})$ has the Bogomolov property. First, he proves that for an element $\alpha \in \Q(E_\text{tor})^\times \setminus \mu_\infty$ plus a correction term is bounded from below. There the bound depends on a prime $p$ that fulfills the conditions of the above Theorem \ref{main}. Precisely, he proves the following.

\begin{proposition}[\cite{MR3090783}, Proposition 6.1]
\label{Philipp6.1}
Suppose $E$ does not have complex multiplication. There exists a constant $c > 0$ depending only on $E$ with the following property. If $\alpha\in \Q(E_{\text{tor}}) \setminus \mu_\infty$ is non-zero, there is a non-zero $\beta \in \overline{\Q}\setminus \mu_\infty$ with $h(\beta) \leq c^{-1} h(\alpha)$ and
\begin{align*}
h(\alpha) + \max \{ 0, \frac{1}{[\Q(\beta):\Q]} \sum_{\tau:\Q(\beta) \to \C} \log |\tau(\beta) - 1|\}  \geq c.
\end{align*}
\end{proposition}

Here, the proof shows that $c = \frac{\log p}{10 p^8}$ is a suitable choice where $p$ is a supersingular and surjective prime.\\

Second, he uses Bilu's equidistribution Theorem in \cite{MR1470340} with a modification of the logarithmic term.\\

We follow this structure and first bound the prime $p$ in Section \ref{supersingularsection}. Here we have to find a small supersingular prime. Although Fouvry and Ram Murty \cite{MR1382477} prove a lower bound for the number of supersingular primes less than or equal to $x$, their bound is not explicit in terms of $E$ so we have to create such a bound. We do so by following Elkies' constructive proof \cite{MR1030140} of the existence of infinitely many supersingular primes for an elliptic curve and making it explicit. We get some congruence relations and put them into one single congruence relation. This allows us to find supersingular primes by finding primes in an arithmetic progression. An unpublished result of Bennett, Martin, O'Bryant and Rechnitzer then gives us an explicit bound for that prime. Although, we also give an effective version where we use Linnik's Theorem \cite{MR2825574}, we actually want to get an explicit result so we cannot use effective versions of Linnik's Theorem for that. Furthermore, we have to give a bound for the biggest non-surjective prime. For that we quote a result of Le Fourn, \cite{LeFourn2016}.\\

Next, we get rid of the sum in Proposition \ref{Philipp6.1}. Instead of modifying the logarithmic term as in \cite{MR3090783} and applying an effective version of Bilu's Theorem, we provide a direct route via a height bound for polynomials due to Mignotte \cite{MR1024420}, see Section \ref{MignotteApproach}.\\

In Section \ref{examples} we give some examples of elliptic curves and their corresponding explicit height bounds.\\

There is also the complementary problem where we do not look at an extension of $\Q$ but at the N\'eron-Tate-height of the elliptic curve $E$ itself. Recall that the N\'eron-Tate-height vanishes precisely at the points of finite order of $E$. Baker \cite{MR1979685} proved that for an elliptic curve $E$ either having complex multiplication or non-integral $j$-invariant, the N\'eron-Tate-height on $E(\Q^{\text{ab}})\setminus E_{\text{tor}}$ is bounded from below. Silverman \cite{MR2029512} proved the same without the constraints on $E$. There are two generalizations of this. First, Baker and Silverman \cite{MR2067482} proved the existence of a lower bound for $A(\Q^{\text{ab}}) \setminus A_{\text{tor}} $ where $A$ is an abelian variety. Second, Habegger \cite{MR3090783} proved that the N\'eron-Tate-height on $E(\Q(E_{\text{tor}})) \setminus E_{\text{tor}} $ is bounded from below. The general conjecture is the following.

\begin{conjecture}[David, mentioned in \cite{MR2533796}]
Let $A$ be an abelian variety defined over a number field $K$ equipped with a N\'eron-Tate-height coming from a symmetric and ample line bundle. Then the N\'eron-Tate-height on $A(K(A_{\text{tor}})) \setminus A_{\text{tor}} $ is bounded from below by a constant only depending on $A/K$ and the definition of the height.
\end{conjecture}

One future task can be making Habegger's bound on the N\'eron-Tate-height explicit. Another one is finding a prime $p$ that is supersingular and surjective for an infinite family of elliptic curves. That would give an unconditional explicit lower bound for the whole family. But while the supersingularity condition can probably be expressed by finitely many congruence relations, finding an unconditional bound for surjective primes is related to an open conjecture of Serre. A possibility may also be looking only at semistable curves and finding an infinite family of semistable curves with the techniques of Kramer \cite{MR715850}.

\subsection*{Acknowledgements}
I am very thankful for all the people who helped me write this article, in particular the following. I thank my advisor Philipp Habegger who proposed this problem to me. He gave me many helpful comments and productive input. I thank the referees of the paper for very helpful and constructive comments. I thank Gabriel Dill for his amazing accurateness and great patience while reading my manuscripts over and over again. I thank Fabien Pazuki for inviting me to talk about my research in Bordeaux. I thank Mike Bennett and his coauthors Martin, O'Bryant and Rechnitzer for sharing their unpublished results with me. This research was done during my PhD at Universit\"at Basel in the DFG project 223746744 "Heights and unlikely intersections".

\section{A supersingular prime for $E$}
\label{supersingularsection}
For exact definitions we refer to \cite{MR2514094} and \cite{MR1312368} and just recall the notation shortly. Let $E$ be an elliptic curve given by a Weierstrass model with rational coefficients. We let $N$ be the conductor (see \cite{MR2514094} VIII.11) and $j_E$ be the $j$-invariant (see \cite{MR2514094} III.1) of $E$. We want to find a small supersingular prime for $E$. In his paper Elkies \cite{MR903384} demonstrates how to find such a prime. We use this technique to find a supersingular prime which we can bound explicitly in terms of constants depending only on $E$. Fouvry and Murty \cite{MR1382477} prove a lower bound for the number of supersingular primes less than or equal to $x$. Their bound is not explicit in terms of $E$. For a prime $p$, we will denote $p^2 = q$ and by $\Q_q$ we will denote the unique unramified quadratic extension of $\Q_p$.\\

For now we consider primes $\ell \equiv 3 \mod 4$ and let $h_\ell$ be the class number of $\Q(\sqrt{-\ell})$. For positive $D$ such that $-D$ is the discriminant of $\Z [\frac{D+\sqrt{-D}}{2}]$ let $P_D$ be the monic polynomial whose roots are (with multiplicity one) exactly the finitely many $j$-invariants of non-isomorphic elliptic curves with complex multiplication by $\Z [\frac{D+\sqrt{-D}}{2}]$. They are polynomials with coefficients in $\Z$ (see \cite{MR903384}). We introduce the convention $\sqrt{-\ell} = i \sqrt{\ell}$ where $\sqrt{\ell}$ is the positive root of $\ell$.\\

Let us also shortly recall - as in Elkie's paper described - that for a prime $p \neq 2$ the roots of $P_D$ in characteristic $p$ are $j$-invariants of curves with complex multiplication by $\mathcal{O}_{D^\prime}$ for some factor $D^\prime$ of $D$ such that $D/D^\prime$ is a perfect square. So for an elliptic curve $E$ with $j$-invariant $j_E$ we have that $p$ divides the numerator of $P_D(j_E)$. If $-D$ is a quadratic non-residue modulo $p$ or the highest power of $p$ dividing $D$ is odd (hence the same holds for $D^\prime$ since $D/D^\prime$ is a perfect square), then $p$ is a supersingular prime for $E$.

We start with a definition.

\begin{definition}[Modular $j$-function]
Let $\tau \in \H$ and let $q = e^{2\pi i\tau}$. We define
\begin{align*}\Delta(\tau) =g_2 (\tau)^3 - 27 g_3(\tau)^2,\end{align*}
where 
\begin{align*}g_2 (\tau) =(2\pi)^4 \frac1{2^2 \cdot 3} (1+240\sum_{n=1}^\infty \sigma_3 (n) q^n)\end{align*}
and
\begin{align*}g_3 (\tau) =(2\pi)^6 \frac1{2^3\cdot 3^2} (1-504\sum_{n=1}^\infty \sigma_5 (n) q^n)\end{align*}
with $\sigma_k (n) = \sum_{d|n} d^k$. Then we call
\begin{align*}j(\tau) = 1728\frac{g_2(\tau)^3}{\Delta(\tau)}\end{align*}
the \emph{modular $j$-function}.
\end{definition}

We need the following lemma.

\begin{lemma}
Let $\ell \equiv 3 \mod 4$. The polynomials $P_\ell$ and $P_{4\ell}$  have odd degree, where $j(\frac{1+\sqrt{-\ell}}{2})$ and $j(\sqrt{-\ell})$ are the only real roots of $P_\ell$ and $P_{4\ell}$, respectively. 
\end{lemma}

\begin{proof}
By \cite{MR0409362}, the map $j: \H / \text{SL}_2 (\Z) \to \C$ is a bijection. One can show that for $\tau \in \H$ with $\Re(\tau) \in \frac12 \Z$, we have that $j(\tau) \in \R$. The same holds for $|\tau| = 1$. Since $j(ai) \to \infty$ as $a\to \infty$ and $j(ai + \frac12) \to -\infty$ as $a\to\infty$, all real values have to be taken on the boundary of the fundamental domain. Since $j$ is bijective on the fundamental domain, all values that $j$ takes on the interior of the fundamental domain have to be complex and not real. By \cite{BHK}, Lemma 2.6 (iii), that is only the case for $\tau = \sqrt{-\ell} $ or $\tau = j(\frac{1+\sqrt{-\ell}}{2})$.
\end{proof}

We now want to find a lower bound $B_E$ such that given an elliptic curve with $j$-invariant $j_E$, for all $\ell \geq B_E$ we have $P_\ell (j_E) > 0$ and $P_{4\ell} (j_E) < 0$. We recall the definition of the modular $j$-function and then find an explicit bound $B_E$.

\begin{lemma}
\label{Lemma1}
Let $E$ be an elliptic curve over $\Q$ with $j$-invariant $j_E$, let \begin{align*}B_E =
\begin{cases}
\left( \frac{\log j_E}{2 \pi}\right)^2 &\mbox{if } j_E>0,\\
\left(\frac{\log |j_E|}{\pi} + 1\right)^2 &\mbox{if } j_E<0,\\
0 &\mbox{if } j_E=0.\\
\end{cases}\end{align*} Then for all primes $\ell > \max\{B_E, 7\}$ such that $\ell \equiv 3 \mod 4$ we have \newline $j(\frac{1+\sqrt{-\ell}}{2}) < j_E < j(\sqrt{-\ell})$, hence $P_\ell (j_E) > 0$ and $P_{4\ell} (j_E) < 0$.
\end{lemma}

\begin{proof}
By the discussion on the real roots of $P_\ell$ and $P_{4\ell}$, we see that \newline $j(\frac{1+\sqrt{-\ell}}{2}) < j_E < j(\sqrt{-\ell})$ implies $P_\ell (j_E) > 0$ and $P_{4\ell} (j_E) < 0$. By \cite{MR0409362}, Theorem 5 on page 249, we have $\Delta(\tau) = (2 \pi)^{12} q \prod_{n=1}^\infty (1-q^n)^{24}$.\\

Let now $\tau = \sqrt{-\ell}$, so $q=e^{2\pi i \sqrt{-\ell}} = e^{-2\pi\sqrt{\ell}}$ with $\ell$ a prime number. Since $\prod_{n=1}^\infty (1-q^n)^{-24} = \prod_{n=1}^\infty \left( \sum_{k=0}^\infty q^{nk} \right)^{24}$ is a product of  geometric series with positive coefficients, it has positive coefficients as a series in $q$ and we get \newline $1728 g_2(\sqrt{-\ell})^3 \Delta(\sqrt{-\ell})^{-1} = \frac{1}{q} \sum_{i=0}^\infty a_i q^i$ with positive integers $a_i$ and $a_0 = 1$. So we get 
\begin{align}
j(\sqrt{-\ell}) = 1728 g_2(\sqrt{-\ell})^3 \Delta(\sqrt{-\ell})^{-1} > \frac{1}{q} = e^{2\pi \sqrt{\ell}}. \label{one}
\end{align}

Let now $\tau = \frac{1+\sqrt{-\ell}}{2}$ so $q=e^{2\pi i \frac{1+\sqrt{-\ell}}{2}} = -e^{-\pi\sqrt{\ell}}$.\\

For $\ell\geq 11$ we have
\begin{align*}
\log \left( \prod_{n=1}^\infty (1-q^n)^{-24} \right) & \geq -24 \sum_{n=1}^\infty \log (1+|q|^n) \\
& \geq -24 \sum_{n=1}^\infty |q|^n \\
& = -24 \frac{|q|}{1-|q|} \\
& > \log (0.99).
\end{align*}
Furthermore, by \cite{MR0409362}, proposition 4 on page 47, we have $$g_2(\frac{1+\sqrt{-\ell}}{2}) = \frac{(2\pi)^4}{12}(1+240\sum_{n=1}^\infty n^3 \frac{q^n}{1-q^n}).$$ Since $\ell \geq 7$ and $\frac{q^n}{1-q^n} \geq \frac{q^n}{1+e^{-\pi\sqrt{7}}} \geq \frac{q^n}{1.0003}$ we have
\begin{align*}
\frac{12}{(2 \pi)^4} g_2 (\frac{1+\sqrt{-\ell}}{2}) &= 1+240\sum_{n=1}^\infty n^3 \frac{q^n}{1-q^n} \\
&\geq 1+239\sum_{n=1}^\infty n^3 q^n\\
&= 1+ 239q\frac{1+4q+q^2}{(1-q)^4}\\
&\geq 0.94.
\end{align*}

We put both inequalities together and get
\begin{align}
\label{two}
j(\frac{1+\sqrt{-\ell}}{2}) &\leq \frac{1}{q} 0.94^3 \cdot 0.99 \leq   -0.82 \cdot e^{\pi \sqrt{\ell}}.
\end{align}
So for $\tau = \sqrt{-\ell}$ (and hence $q$ positive) we get
\begin{align}
0 &< e^{2\pi \sqrt{\ell}} = \frac{1}{q}  < j(\sqrt{-\ell})\label{first1}
\end{align}
and for $\tau = \frac{1+\sqrt{-\ell}}{2}$ (and hence $q$ negative) we get
\begin{align}
j(\frac{1+\sqrt{-\ell}}{2}) \leq 0.82\frac{1}{q} = -0.82 \cdot e^{\pi \sqrt{\ell}} < 0 \label{second}.
\end{align}

For $j_E \geq 0$ the inequality $j(\frac{1+\sqrt{-\ell}}2) < j_E$ holds true by equation \eqref{second} and for $j_E \leq 0$ the inequality $j_E < j(\sqrt{-\ell})$ holds true by equation \eqref{first1}. So we can take $B_E = 0$ if $j_E = 0$. To complete the proof we may hence assume $j_E \neq 0$ and it suffices to show $$j_E < e^{2\pi\sqrt{\ell}} \text{ if } j_E > 0$$ and $$j_E > - 0.82 \cdot e^{\pi\sqrt{\ell}}\text{ if } j_E < 0.$$
The first inequality follows from $$\left( \frac{\log j_E}{2\pi} \right)^2 < \ell$$ and the second one from $$\left( \frac{\log |j_e|}{\pi} - \frac{\log 0.82}{\pi} \right)^2 < \ell$$
and we proved the statement.
\end{proof}

By Elkies (\cite{MR903384}), we know that we can find a supersingular prime for $E$ by taking a prime $\ell$ such that $\Legendre{-1}{\ell} = -1$ and $\Legendre{p}{\ell} = + 1$ for every prime $p$ of bad reduction. Since $j_E$ is a rational number, $P_\ell (j_E)P_{4\ell} (j_E)$ is also rational and it makes sense to speak of numerators and denominators. Then the factorization of the numerator of either $P_\ell (j_E)$ or $P_{4\ell} (j_E)$ contains a supersingular prime for $E$. So if we can find such an $\ell$ and bound the numerator of $P_\ell (j_E)P_{4\ell} (j_E)$, we also get a bound for a supersingular prime for $E$.

We start by bounding $N_\ell$.

\begin{lemma}[Fouvry-Ram Murty, \cite{MR1382477}, Lemma 5]
\label{FouvryMurty}
For $\ell \equiv 3 \mod 4$ prime and $C = 10^{10} \log (|j_E| + 745)$ we have
\begin{align*}| P_{\ell} (j_E) P_{4\ell} (j_E) | \leq  e^{3 C\sqrt{\ell} (\log \ell)^2 + 4h_\ell}.\end{align*}
\end{lemma}

\begin{proof}
For an integer $a\neq 0$, let $\omega (a)$ be the number of its distinct prime divisors. We use the following inequalities from \cite{MR993343} Lemmas 5 and 6: 
\begin{align*}
|P_\ell (j_E)| \leq 2^{h_\ell} e^{ \log (|j_E| + 745) \sqrt{\ell} \sum_{1\leq a\leq\sqrt{\ell}} \frac{2^{\omega(a)}}{a}}
\end{align*}
and
\begin{align*}
|P_{4\ell} (j_E)| \leq 2^{3 h_\ell} e^{4 \log (|j_E| + 745) \sqrt{\ell} \sum_{1\leq a\leq\sqrt{\ell}} \frac{2^ \omega(a)}{a}}.
\end{align*}
Now we follow the proof of Fouvry and Ram Murty:
\begin{align*}
\sum_{1\leq a\leq\sqrt{\ell}} \frac{2^{\omega(a)}}{a} &\leq \prod_{p \text{ prime }\leq\sqrt{\ell}} \left(1+ \frac{2}{p} + \frac{2}{p^2} + \frac{2}{p^3} + ...\right) \\
&= \prod_{p\leq\sqrt{\ell}} \left(1+ \frac{2}{p} \sum_{k=0}^\infty \frac{1}{p^k}\right) \\
&= \prod_{p\leq\sqrt{\ell}} \left(1+ \frac{2}{p-1}\right).
\end{align*}
Now we take the logarithm.
\begin{align*}
\log \left( \prod_{p\leq \sqrt{\ell}} \left(1+ \frac{2}{p-1}\right) \right)& \leq \sum_{p\leq \sqrt{\ell}} \frac{2}{p-1}  \\
& = 2 \sum_{p\leq \sqrt{\ell}} \left( \frac{1}{p} + \frac{1}{(p-1)p}  \right) & \\
& \leq 2 \sum_{p\leq \sqrt{\ell}} \left( \frac{1}{p} + \frac{2}{p^2}  \right)
\end{align*}
since $\frac{p}{2} \leq p-1$. Then
\begin{align*}
\log \left( \prod_{p\leq \sqrt{\ell}} \left(1+ \frac{2}{p-1}\right) \right) & \leq 2 \sum_{p\leq \sqrt{\ell}}\frac{1}{p} + 4 \sum_{n=1}^\infty \frac{1}{n^2}  & \\
& \leq 2 \log\log \sqrt{\ell} + 0.523 + \frac{2}{(\log \sqrt{\ell})^2} + 2 \frac{\pi^2}{3}
\end{align*}
(see  \cite{MR0137689}, Cor after Thm 5 for the last estimate).
So we get
\begin{align*}
\sum_{1\leq a\leq\sqrt{\ell}} \frac{2^{\omega(a)}}{a} & \leq \frac{1}{4} e^{0.523+\frac{2}{(\log \sqrt{2})^2} + 2 \frac{\pi^2}{3}}(\log \ell)^2\\
& \leq 5176493608(\log \ell)^2\\
& \leq 5.2 \cdot 10^9 (\log \ell)^2.
\end{align*}
And as a result
\begin{align*}
| P_\ell (j_E) P_{4\ell} (j_E)| & \leq 2^{h_\ell} e^{5.2\cdot 10^9 \log (|j_E| + 745) \sqrt{\ell} (\log \ell)^2} 2^{3 h_\ell} e^{2.08\cdot 10^{10} \log (|j_E| + 745) \sqrt{\ell} (\log \ell)^2}\\
& \leq e^{3 C\sqrt{\ell} (\log \ell)^2 + 4h_\ell} .
\end{align*}
\end{proof}

For ease of notation let $N_\ell := -\num(P_\ell (j_E) P_{4\ell} (j_E))$.

\begin{lemma}
\label{L:num}
With $\ell \equiv 3 \mod 4$ prime and $\ell\geq 5$ we have
\begin{align*}|N_\ell| \leq e^{2.3 \cdot 10^{11} \sqrt{\ell} (\log \ell)^2 \max(\log 2, h(j_E))}.\end{align*}
\end{lemma}

\begin{proof}
By \cite{MR1312368}, App C, Prop. 11.1, we know $\deg P_\ell = h_\ell$ and by \cite{MR594936}, p. 217 thm 2, we have $\deg P_\ell = \deg P_{4\ell}$.\\

Furthermore, by \cite[Theorem 12.10.1]{MR665428} and \cite[Theorem 12.13.3]{MR665428} and since $\ell \geq 7$ we can use the class number formula to get the following bound:
\begin{align*}h_\ell \leq \frac{2\sqrt{\ell}}{2 \pi}(2+\log{\ell}) \leq \frac{3\sqrt{\ell}\log \ell}{\pi}.\end{align*}
So\begin{align*}\deg (P_\ell P_{4\ell}) = 2 h_\ell \leq \frac{6\sqrt{\ell} \log{\ell}}{\pi}\end{align*}
 so using Lemma $\ref{FouvryMurty}$ we get
\begin{align*}|P_\ell (j_E) P_{4\ell} (j_E)| \leq e^{3 C\sqrt{\ell} (\log \ell)^2 + 4h_\ell}  \leq e^{3 C\sqrt{\ell} (\log \ell)^2 + \frac{12\sqrt{\ell} \log{\ell}}{\pi}} .\end{align*}
Now we can bound the numerator of $P_\ell (j_E) P_{4\ell} (j_E)$. 

With $\max(\log 2, h(j_E)) \geq \log 2 > 0$ and $\log (\den (j_E)) \leq \max(\log 2, h(j_E)) $ we get 
\begin{align*}
\log |N_\ell| & \leq  3C  \sqrt{\ell} (\log \ell)^2 + \frac{12\sqrt{\ell} \log{\ell}}{\pi} + \log ( |\den (j_E)|^{\frac{6\sqrt{\ell} \log{\ell}}{\pi}})\\
& = 3C  \sqrt{\ell} (\log \ell)^2 + \frac{12\sqrt{\ell} \log{\ell}}{\pi} +\frac{6\sqrt{\ell} \log{\ell}}{\pi} \log |\den (j_E)|\\
& = 3C  \sqrt{\ell} (\log \ell)^2 + \frac{12\sqrt{\ell} \log{\ell}}{\pi} +\frac{6\sqrt{\ell} \log{\ell}}{\pi} \max(\log 2, h(j_E))\\
& \leq \sqrt{\ell} (\log \ell)^2 (3C + 4 + 2 \max(\log 2, h(j_E)))\\
& = \sqrt{\ell} (\log \ell)^2 (3\cdot 10^{10} \log (|j_E| +745) + 4 + 2 \max(\log 2, h(j_E)))\\
& \leq \sqrt{\ell} (\log \ell)^2 (3\cdot 10^{10} \max(\log 2,h(j_E)) + 3 \cdot 10^{10} \log 745 + 4 + 2 \max(\log 2,h(j_E)))\\
& = \sqrt{\ell} (\log \ell)^2 ((3\cdot 10^{10} + 2) \max(\log 2,h(j_E)) + 3 \cdot 10^{10} \log 745 + 4)\\
& = \sqrt{\ell} (\log \ell)^2 ((3 \cdot 10^{10} + 2) \max(\log 2,h(j_E)) + (3 \cdot 10^{10} \log 745 + 4)\frac{\log 2}{\log 2})\\
& \leq \sqrt{\ell} (\log \ell)^2 ((3\cdot 10^{10} + 2) \max(\log 2,h(j_E)) + (3 \cdot 10^{10} \log 745 + 4)\frac{\max(\log 2,h(j_E))}{\log 2})\\
& = \sqrt{\ell} (\log \ell)^2 (3\cdot 10^{10} + 2 + \frac{3 \cdot 10^{10} \log 745 + 4}{\log 2}) \max(\log 2,h(j_E))\\
& \leq \sqrt{\ell} (\log \ell)^2 (3\cdot 10^{10} + 2 + 2.004 \cdot 10^{11}) \max(\log 2,h(j_E))\\
& \leq \sqrt{\ell} (\log \ell)^2  2.4 \cdot 10^{11} \max(\log 2,h(j_E)),
\end{align*}
which is what we wanted to show.
\end{proof}

Now we can use the following explicit bound for primes in arithmetic progressions to bound $\ell$ and hence get an estimate for $p$.

\begin{theorem}[Theorem 1.2, \cite{2018arXiv180200085B}]
Let $q\geq 3$ and $\gcd (a,q) =1$. There exist explicit positive constants $c_\theta (q)$ and $x_\theta (q)$ such that
\begin{align*}
|\theta (x;q,a) - \frac{x}{\varphi(q)}| < c_\theta (q) \frac{x}{\log x} \text{ for all } x \geq x_\theta (q)
\end{align*}
where $\theta(x;q,a) = \sum_{p\leq x, p \equiv a \mod q} \log p$ and $\varphi$ is Euler's $\varphi$-function. Moreover,
\begin{align*}
c_\theta (q) \leq \frac{1}{3600} \frac{q}{\varphi (q)},
\end{align*}
while $x_\theta (q)$ satisfies $x_\theta (q) < x_0 (q)$ where
\begin{align}
x_0 (q) = \begin{cases} 7.94 \cdot 10^9, &\text{if } 3 \leq q\leq 600\\
4.81 \cdot \frac{10^{12}}{q}, &\text{if } 601 \leq q \leq 10^5\\
\exp (0.036 \sqrt{q} (\log q)^3), &\text{if } q > 10^5 \end{cases}.
\end{align}
\end{theorem}

We can derive the following corollary.

\begin{corollary}
\label{Bennett}
Let $q > 10^5$ and $a$ be coprime positive integers. Then there exists a prime $p \equiv a \mod q$ with $p\leq \exp (0.036 \sqrt{q} (\log q)^3) $.
\end{corollary}

\begin{proof}
Consider $\theta(x;q,a)$ for $x \geq \exp (0.036 \sqrt{q} (\log q)^3) $ and remark that then also $x \geq \exp (\frac{q}{3600})$ holds true. Assume $\theta(x;q,a)  = 0$. Then we have $|\theta(x;q,a) - \frac{x}{\varphi(q)}| = \frac{x}{\varphi(q)}$. The theorem above tells us that for $x \geq \exp (0.036 \sqrt{q} (\log q)^3)$ we have $\frac{x}{\varphi(q)} < \frac1{3600} \frac{q}{\varphi(q)}$ which is equivalent to $x < \exp (\frac{q}{3600})$. This is a contradiction, so $\theta(x;q,a)$ cannot be zero.
\end{proof}
Now we can turn to our theorem.

\begin{theorem}
\label{boundp}
Let $E$ be an elliptic curve with $j$-invariant $j_E$ and conductor $N$. Let $B_E$ be as in Lemma \ref{Lemma1}, $M \in\N$ and $n= \max (11, M, B_E)$. Then there exists a supersingular prime $p$ of $E$ such that $p \geq n$ and
\begin{align*}\log p \leq 2.5 \cdot 10^{9} e^{0.018 \sqrt{8 N e^{\vartheta (n)}} (\log ( 8 N e^{\vartheta (n)}))^3} N e^{\vartheta (n)} (\log ( 8 N e^{\vartheta (n)}))^6 \max (\log 2, h(j_E)).\end{align*}
\end{theorem}

\begin{proof}
In this proof we follow Elkies' construction of supersingular primes in his paper \cite{MR903384}.\\

Let us assume as usual that $\ell\equiv 3\mod 4$. By the proposition in the said paper, we know that the product $P_\ell P_{4\ell}$ is a square modulo $\ell$. Since $P_\ell$ and $P_{4\ell}$ are monic, also their product $P_\ell P_{4\ell}$ is monic. Hence both the numerator and the denominator of $P_\ell P_{4\ell}$ have to be squares modulo $\ell$. We already proved that for every $\ell \geq B_E$ as in Lemma $\ref{Lemma1}$, the numerator of $P_\ell(j_E)P_{4\ell} (j_E)$ is a negative integer, so $N_\ell$ is a positive integer.

Now we want to construct and bound $\ell$ such that: every prime $p$ with bad reduction is a square modulo $\ell$, $\ell$ is congruent to $7$ modulo $8$ and $\ell$ is at least as large as $\max (5, M, B_E)$. We find that $N_\ell$ has a prime divisor $p$ with $p=\ell$ or $\Legendre{p}{\ell} = -1$. This must be a supersingular prime for $E$ by Elkies \cite{MR903384}. By adding more congruence conditions $\Legendre{p^\prime_i}{\ell} = 1$ for finitely many primes $p^\prime_i$, we can rule out that $\ell$ is in a finite prescribed set. Since we want to exclude all number less than or equal to $n$, we add the condition $\Legendre{p^\prime_i}{\ell} = 1$ for all $p_i^\prime \leq n$.

With the Chinese Remainder Theorem and Gauss' Reciprocity Law we can put the equations 
\begin{align*}
\Legendre{p_i}{\ell} &= 1 \text{ for all primes } p_i \mid \rad (6 N)\\
\Legendre{p_i^\prime}{\ell} &= 1 \text{ for all primes } p_i^\prime \leq n\\
\text{and } \ell &\equiv 7 \mod 8\\
\end{align*}
into one equation
\begin{align}
\ell \equiv a \mod q \label{amodq}
\end{align}
for some $a$ which is coprime to $q$ with $24 \leq q \leq 8 \rad(N) e^{\vartheta (n)} \leq 8 N e^{\vartheta (n)}$.

By Corollary \ref{Bennett} and with $8 N e^{\vartheta (n)} > 10^5$ (this is true since $n$ and $N$ are both at least $11$) we know that there is a prime $\ell$ satisfying $\ell\equiv a \mod q$ with
\begin{align*}
\ell &\leq e^{0.036 \sqrt{8 N e^{\vartheta (n)}} (\log ( 8 N e^{\vartheta (n)}))^3}.
\end{align*}

Together with Lemma \ref{L:num} this gives us a supersingular prime $p$ which is bounded from above by
\begin{align*}
p & \leq e^{2.4 \cdot 10^{11} \sqrt{\ell} (\log \ell)^2 \max(\log 2,h(j_E))}.
\end{align*}
Consider
\begin{align*}
\sqrt{\ell} (\log \ell)^2 & \leq \sqrt{e^{0.036 \sqrt{8 N e^{\vartheta (n)}} (\log ( 8 N e^{\vartheta (n)}))^3}} (\log (e^{0.036 \cdot \sqrt{8 N e^{\vartheta (n)}} (\log ( 8 N e^{\vartheta (n)}))^3}))^2 \\
& \leq e^{0.018 \sqrt{8 N e^{\vartheta (n)}} (\log ( 8 N e^{\vartheta (n)}))^3} 0.0104 N e^{\vartheta (n)} (\log ( 8 N e^{\vartheta (n)}))^6.
\end{align*}
For better readability we take the logarithm
\begin{align*}
\log p & \leq 2.4 \cdot 10^{11} \sqrt{\ell} (\log \ell)^2 \max(\log 2,h(j_E))\\
& \leq 2.4 \cdot 10^{11} e^{0.018 \sqrt{8 N e^{\vartheta (n)}} (\log ( 8 N e^{\vartheta (n)}))^3} 0.0104 N e^{\vartheta (n)} (\log ( 8 N e^{\vartheta (n)}))^6 \max (\log 2, h(j_E))\\
& \leq 2.5 \cdot 10^{9} e^{0.018 \sqrt{8 N e^{\vartheta (n)}} (\log ( 8 N e^{\vartheta (n)}))^3} N e^{\vartheta (n)} (\log ( 8 N e^{\vartheta (n)}))^6 \max (\log 2, h(j_E)).
\end{align*}
which is what we wanted to prove.
\end{proof}

If one does not attach importance to explicit constants, we can also use Linnik's Theorem with an explicit exponent as proved by Xylouris \cite{MR3086819} in Theorem 2.1. We get the following better bound.

\begin{corollary}
\label{effectiveElkies}
With the notation from the theorem there exists an effectively computable constant $c$ such that
\begin{align*}
\log p \leq c q^{5/2} (\log q)^2 \max(\log 2,h(j_E)).
\end{align*}
\end{corollary}

\begin{proof}
We go back to the proof of the theorem before and replace the part where we use the explicit result on primes in arithmetic progressions by Xylouris' effective version of Linnik's Theorem (equation \eqref{amodq}). It gives us
\begin{align*}
\ell \leq c^\prime q^5
\end{align*}
with an effective constant $c^\prime$.
So we get
\begin{align*}
\log p & \leq 2.3 \cdot 10^{11} \sqrt{c^\prime\cdot q^5}(\log (c^\prime q^5))^2 \max(\log 2,h(j_E))\\
 & \leq c q^{5/2} (\log q)^2 \max(\log 2,h(j_E)).
\end{align*}
\end{proof}

With a result of von K\"anel we can bound the height of the $j$-invariant by the conductor.

\begin{theorem}[\cite{MR3296485}, equation 3.5 and \cite{2016arXiv160506079V}, Proposition 6.8]
\label{vK}
Let $E$ be an elliptic curve over $\Q$ with $j$-invariant $j_E$ and conductor $N$. Then we have
\begin{align*}
h(j_E) & \leq 12 h_E + 6 \log \max (1, h_E) + 75.84
\end{align*}
where
\begin{align*}
h_E \leq \frac{N}{12} \log N + \frac{N}{32} \log \log \log N + \frac{N}{18} + 2 \pi + \frac12 \log \frac{163}\pi.
\end{align*}
\end{theorem}

\begin{corollary}
Let $E$ be an elliptic curve over $\Q$ with $j$-invariant $j_E$ and conductor $N$. Then we have
\begin{align*}
\max(\log 2,h(j_E)) \leq 10 N \log N.
\end{align*}
\end{corollary}

\begin{proof}
Since $\max(\log 2,h(j_E)) $ differs from $h(j_E)$ only when $h(j_E) = 0$ and since $18N\log N$ is always greater than $\log 2$ (since $N  > 1$) it is enough to show that $h(j_E) \leq 10N\log N$.

We want to simplify the bound from Theorem \ref{vK} and use the fact that the conductor $N$ of an elliptic curve over $\Q$ is at least $11$. We get
\begin{align*}
12 h_E &\leq N \log N + \frac{3 N}{8} \log \log \log N + \frac{2 N}{3} + 99.1\\
&\leq N \log N + \frac{3 N}{8} \log N + \frac{2 N}{3} + 99.1\\
&\leq N \log N + \frac38 N \log N + \frac{2 }{3\log 11} N \log N + 99.1\\
& \leq 1.66 N \log N + 99.1\\
& \leq 5.42 N \log N
\end{align*}
and
\begin{align*}
6 \log h_E & \leq 6 \log (\frac{10}{12} N \log N)\\
& \leq 6 \log N^2\\
& \leq \frac6{11} N \log N^2\\
& \leq \frac{12}{11} N \log N.
\end{align*}
Altogether we get
\begin{align*}
h(j_E) & \leq 12 h_E + 6 \log \max (1, h_E) + 75.84 \\
& \leq 10 N \log N + \frac{12}{11} N\log N + 75.84 \\
& \leq 9.39 N \log N\\
& \leq 10 N \log N.
\end{align*}
which is the desired bound.
\end{proof}

Now we can reformulate our result with dependence only on the conductor.

\begin{theorem}
\label{boundpnoj}
Let $E$ be an elliptic curve with $j$-invariant $j_E$ and conductor $N$. Let $M \in\N$, $q = 4 \rad (6 N)$ and $n= \max (M, (6 N\log N)^2)$. Then there exists a supersingular prime $p$ of $E$ such that $p \geq n$ and
\begin{align*}\log p \leq 2.5 \cdot 10^{10} e^{0.018 \sqrt{8 N e^{\vartheta (n)}} (\log ( 8 N e^{\vartheta (n)}))^3} N e^{\vartheta (n)} (\log ( 8 N e^{\vartheta (n)}))^6 N \log N.\end{align*}
\end{theorem}

\begin{proof}
First, we prove that $B_{j_E} \leq \left(6 N\log N\right)^2$.

\begin{align*}
B_{j_E} & \leq \left(\frac{\log |j_E|}{\pi} + 1\right)^2 \\
& \leq \left(\frac{\max(\log 2,h(j_E))}{\pi} + 1\right)^2\\
& \leq \left(\frac{10 N\log N}{\pi} + 1\right)^2\\
& \leq (6 N\log N)^2.
\end{align*}
With the bound for the height of the $j$-invariant from the above corollary we get the desired bound.
\end{proof} 
\section{Handling the sum}

\subsection{Explicit Bounds and Algebraic Numbers}
\label{MignotteApproach}
In this section we want to bound the sum in Proposition \ref{Philipp6.1} from before. Our goal is to eventually show that this is negligible when compared to $\frac{\log p}{2p^8}$. This section is not dependent on elliptic curves, it is only about algebraic numbers. We start with the following Lemma.

\begin{lemma}
\label{sum}
Let $\beta \in \overline{\Q}^\times \setminus \{1\}$ of degree $[\Q(\beta):\Q] = d \geq 2$ and let $0 < \varepsilon < \frac12$. Then
\begin{equation*}
\frac{1}{[\Q(\beta):\Q]} \sum_{\tau : \Q (\beta) \hookrightarrow \C} \log |\tau (\beta) -1 | \leq 2(\varepsilon |\log \varepsilon| + |\log (1-\varepsilon)|)+\frac2{\varepsilon d} \log d  + (1+\frac{1}{\varepsilon}) h(\beta),\\
\end{equation*}
where $\tau$ runs over all embeddings of $\Q (\beta)$ into $\C$.
\end{lemma}

\begin{proof}
Let $F(x) = a_d x^d + \cdots + a_0 = a_d \cdot (x-\beta_1)\cdot \cdots \cdot (x-\beta_d)$ be the unique integral polynomial of degree $d=[\Q(\beta):\Q]$ that vanishes at $\beta$ with $a_d \geq 1$ and $a_0,...,a_d$ coprime. Remark that $\beta_i \in \C$.
Since
\begin{align*}0 \neq |F(1)| = |a_d| \cdot \prod_{i=1}^d |\beta_i - 1| \end{align*}
we get
\begin{align}
\frac1d \log |F(1)| &= \frac{\log |a_d|}{d} + \frac1d \sum_{i=1}^d \log |\beta_i - 1| \nonumber\\
&\geq \frac1d \sum_{i=1}^d \log |\beta_i - 1| \nonumber\\
&= \frac{1}{[\Q(\beta):\Q]} \sum_{\tau : \Q (\beta) \hookrightarrow \C} \log |\tau (\beta) -1 |. \label{betaUmformung}
\end{align}
So it is enough to bound $|F(1)|$ in order to prove the Lemma.\\

For any polynomial $G = g_n x^n + ... + g_0 \in \Z [x]$ we define its height as $H(G) := \max_i |g _i|$. Furthermore, let $G_k := \frac{d^k G}{k! d x^k} = \sum_{i=k}^n \binom{i}{k} g_i x^{i-k} \in\Z[x]$ and $D \geq d$. We fix $D$ later in terms of $\varepsilon$ and $d$. By Mignotte's Theorem B in \cite{MR1024420} we can find a polynomial $A(x) = \sum_{i=0}^{D-d} a_i x^i \in\Z[x]\setminus \{0\}$ of degree at most $D - d$ such that
\begin{align}
H(A \cdot F) \leq ((D+1)^{d/2} H(\beta)^{D d})^\frac{1}{D+1-d} . \label{Mignotte}
\end{align}
Let $k\in \N_0$ be the multiplicity of the zero at $1$ of $A$. Since the degree of $A$ is at most $D-d$ we have $k\leq D-d$. Then $A_{k-i} (1) = 0$ for all positive $i \leq k$ and $A_k (1) \neq 0$. As $A_k (1) \in \Z$ we find $|A_k (1)| \geq 1$ and thus by the Leibniz rule we get
\begin{align}
|F(1)| &\leq |A_k (1)| |F(1)| \nonumber\\
&= |(A \cdot F)_k (1)| \nonumber\\
&\leq (D-k+1)H((A \cdot F)_k)\nonumber\\
&\leq (D-k+1) \binom{D}{k} H(A \cdot F) \label{3.9}.
\end{align}
By putting inequalities \eqref{betaUmformung}, \eqref{Mignotte} and \eqref{3.9} together we get
\begin{align*}
\frac{1}{[\Q(\beta):\Q]} \sum_{\tau : \Q (\beta) \hookrightarrow \C} \log |\tau (\beta) -1 | &\leq \frac1d \log |F(1)| \\
& \leq  \frac1d \log \left((D-k+1) \binom{D}{k} H(A \cdot F)\right) \\
&\leq \frac1d \log \left( (D-k+1) \binom{D}{k} ((D+1)^{d/2} H(\beta)^{D d})^\frac{1}{D+1-d} \right)\\
&\leq \frac1d \log \left( \binom{D}{k} (D+1)^{\frac{d}{2(D+1-d)}+1} H(\beta)^\frac{Dd}{D+1-d} \right).
\end{align*}
The right hand side equals
\begin{align*}
\frac1d \log \binom{D}{k} + \left(\frac{1}{2(D+1-d)}+\frac1d\right) \log (D+1) + \frac{D}{D+1-d} h(\beta).\end{align*}
Note that $\varepsilon d\leq \varepsilon \lfloor(1+\varepsilon)d\rfloor$ and so with $D := \lfloor(1+\varepsilon) d \rfloor$ we have
\begin{align*}
k\leq D-d \leq \varepsilon d \leq \varepsilon \lfloor(1+\varepsilon)d\rfloor.
\end{align*} So we can apply Lemma 16.19 of \cite{MR2238686} with $q = \varepsilon > 0$ and $n=D$. We get $\binom{\lfloor(1+\varepsilon) d\rfloor}{k} \leq 2^{-(1+\varepsilon) d(\varepsilon\log \varepsilon + (1-\varepsilon)\log(1-\varepsilon))}$. Since $\varepsilon < 1$ we can write $|\log \varepsilon|$ instead of $- \log \varepsilon$ and $|\log (1 - \varepsilon)|$ instead of $- \log (1 - \varepsilon)$. So we can bound the above expression by
\begin{equation*}((1+\varepsilon)\varepsilon |\log \varepsilon| + (1-\varepsilon^2)|\log(1-\varepsilon)|)\log 2+\frac{1+2\varepsilon}{2\varepsilon d} \log ((1+\varepsilon)d + 1) + (1+\frac{1}{\varepsilon}) h(\beta).\end{equation*}
We start by bounding the first part:
\begin{align*}
((1+\varepsilon)\varepsilon |\log \varepsilon| + (1-\varepsilon^2) |\log(1-\varepsilon)|)\log 2 & \leq (\frac32 \varepsilon |\log \varepsilon| + |\log(1-\varepsilon)|)\log 2 \\
& \leq 2(\varepsilon |\log \varepsilon| + |\log (1-\varepsilon)|).
\end{align*}
The second summand can also be bounded further:
\begin{align*}
\frac{1+2\varepsilon}{2\varepsilon d} \log ((1+\varepsilon)d + 1) &\leq  \frac{2}{2\varepsilon d} \log (d^2) \\
& = \frac2{\varepsilon d} \log d.
\end{align*}
When we put both bounds together we get
\begin{align}
2(\varepsilon |\log \varepsilon| + |\log (1-\varepsilon)|)+\frac2{\varepsilon d} \log d + (1+\frac{1}{\varepsilon}) h(\beta)\label{epsilon}
\end{align}
as an upper bound for $\frac{1}{[\Q(\beta):\Q]} \sum_{\tau : \Q (\beta) \hookrightarrow \C} \log |\tau (\beta) -1 |$.

\end{proof}

Later, we fix an $\varepsilon$ and then get an explicit bound. But first, we want to look at the terms separately.

\begin{lemma}
\label{L:1}
Let $0<x\leq \frac12$. Then
\begin{align*}-2(x \log x + \log (1-x))\leq - x \log x(2+\frac4{\log 2}).\end{align*}
\end{lemma}

\begin{proof}
We have $\log (1+t) \leq t$ for all $t\geq 0$ and $$-\log (1-x) = \log \frac1{1-x} = \log (1+\frac{x}{1-x}).$$ So $$-\log (1-x) \leq \frac{x}{1-x} \leq 2x$$ since $x\leq \frac12$. The bound then follows from $2x \leq -2 \frac{x\log x}{\log 2}$ as $-\frac{\log x}{\log 2} \geq 1$.
\end{proof}

For our purpose the following corollary is sufficient.

\begin{corollary}
\label{C:1}
Let $0<x\leq \frac12$ and $0<\gamma <1$. We have
\begin{align*}-2(x \log x + \log (1-x))\leq 8 \frac1{\gamma e} x^{1-\gamma}.\end{align*}
\end{corollary}

\begin{proof}
We use the lemma from above and want to show that $-x\log x \leq \frac1{\gamma e}x^{1-\gamma}$. We use basic calculus to get the maximum value. We compute the derivative with respect to $x$ as
\begin{align*}\left(- x^\gamma \log x \right)^\prime &= -x^{\gamma-1}(\gamma \log x + 1).\end{align*}

In our interval, this is zero if and only if
\begin{align*}
x = e^{-\frac1\gamma}.
\end{align*}
Since we have $-(\frac12)^{\gamma-1}(\gamma \log \frac12 + 1) < 0$ for all $0<\gamma <1$, the slope of $- x^\gamma \log x$ changes its sign at $x = e^{-\frac1\gamma}$ from positive to negative and so we have a maximum. Finally, we have
\begin{align*}
- x^\gamma \log x &\leq -e^{-\frac1\gamma \gamma} \log(e^{-\frac1\gamma}) = \frac1{\gamma e}
\end{align*}
which after multiplying by the positive value ${x^{1-\gamma}}$ gives the desired result.

\end{proof}

We need a similar result for the second summand.
\begin{lemma}
\label{L:2}
Let $0<\eta<1$ and $d\geq 16$. Then for every $x > \frac1{4d} \left( \frac{\log\log d}{\log d}\right)^3$ we have
\begin{align*}\frac{\log d}{d} \leq  \frac{19}{\eta^4} x^{1-\eta}.\end{align*}
\end{lemma}

\begin{proof}
Let us look at the function $4\frac{(\log d)^4}{d^{\eta}}$. To see that it is bounded from above we compute the derivative with respect to $d$:
\begin{align*}
\left(4\frac{(\log d)^4}{d^\eta}\right)^\prime = \frac4{d^{1+\eta}} (\log d)^3 (4 - \eta\log d).
\end{align*}
This is zero if and only if $d = e^{4/\eta}$. Since the derivative is positive for $d < e^{4/\eta}$, our extremum is a maximum. So $4\frac{(\log e^{4/\eta})^4}{(e^{4/\eta})^{\eta}} = \frac{4^5}{e^4\eta^4}$ is an upper bound for $4\frac{(\log d)^4}{d^\eta}$ and we get
\begin{align*}
\frac{19}{\eta^4} & \geq \frac{4^5}{e^4\eta^4}\\
& \geq 4\frac{(\log d)^4}{d^{\eta}}\\
& \geq 4^{1-\eta}\frac{(\log d)^{4-3\eta}}{d^\eta(\log\log d)^{3-3\eta}}\\
& = \frac{\log d}{d} \left( 4d \left( \frac{\log d}{\log \log d} \right)^3\right)^{1 - \eta} \\
& \geq \frac{\log d}{d} \left(\frac1{x}\right)^{1-\eta}
\end{align*}
which gives the desired inequality.

\end{proof}

In the next lemma we combine all of the previous results of this section.

\begin{lemma}
\label{sumexpl}
Let $0 < \delta < \frac12$ and let $\beta \in \overline{\Q}^\times \setminus \mu_\infty$ be such that 
$[\Q(\beta):\Q] \geq 16$ and $h(\beta)^{1/2} \leq \frac12$. Then we have
\begin{align}
\frac{1}{[\Q(\beta):\Q]} \sum_{\tau : \Q (\beta) \hookrightarrow \C} \log |\tau (\beta) -1 | \leq \frac{40}{\delta^4} h(\beta)^{\frac12-\delta}.\label{boundthesuminexpl}
\end{align}
\end{lemma}

\begin{proof}
Set $\varepsilon = h(\beta)^{1/2}$. Then Lemma \ref{sum} gives
\begin{align*}
\frac{1}{[\Q(\beta):\Q]} \sum_{\tau : \Q (\beta) \hookrightarrow \C} \log |\tau (\beta) -1 | \leq - 2 (\varepsilon \log \varepsilon + \log (1-\varepsilon))+\frac{2}{\varepsilon d} \log d + (1+\frac{1}{\varepsilon}) h(\beta)\\
\leq - 2 (h(\beta)^{1/2} \log h(\beta)^{1/2} + \log (1-h(\beta)^{1/2}))+\frac{2 \log d}{h(\beta)^{1/2} d} + h(\beta)^{1/2} + h(\beta)^{1/2}.
\end{align*}
Now since $h(\beta)^{1/2} \leq \frac12$, we can apply Corollary \ref{C:1} to the first term. By the main theorem of \cite{MR1367580} we also have $h(\beta) > \frac{1}{4d} \left( \frac{\log \log d}{\log d} \right)^3$ and so we can apply Lemma \ref{L:2} to the third term and for any $0 < \gamma,\eta < 1$ we get:
\begin{align*}
\frac{1}{[\Q(\beta):\Q]} \sum_{\tau : \Q (\beta) \hookrightarrow \C} \log |\tau (\beta) -1 |&\leq \frac8{\gamma e} h(\beta)^{\frac12(1-\gamma)} + \frac{38}{\eta^4} h(\beta)^{\frac12-\eta} + 2h(\beta)^{1/2}.
\end{align*}
Now we set $\gamma := 2 \delta$ and $\eta := \delta$ and get
\begin{align*}
\frac8{\gamma e} h(\beta)^{\frac12(1-\gamma)} + \frac{38}{\eta^4} h(\beta)^{\frac12-\eta} + 2h(\beta)^{1/2} & = \frac1{\delta^4} h(\beta)^{\frac12 -\delta} (\frac8{2e}\delta^3 + 38+2\delta^4) \\
&\leq \frac{40}{\delta^4} h(\beta)^{\frac12 -\delta},
\end{align*}
which is what we wanted to show.
\end{proof}

\subsection{An alternative approach to handle the sum}
\label{altapp}
As the title of this section already reveals, we want to give an alternative approach to handle the sum in Proposition \ref{Philipp6.1}. This approach was communicated by Amoroso and appears in \cite{MR3182009}. We do not bound the sum directly but we try to get rid of the sum before it even occurs. For this we quote and try to improve a result of Habegger. Recall the notations and conventions of Chapter \ref{prel}.

The following lemma is the result we want to improve.

\begin{lemma}[\cite{MR3090783}, Lemma 4.2]
\label{Lemma4.2}
Let $N\in\N$ and suppose $p|N$ and $\alpha \in \Q_q (N)^\times$. Then for all $\psi \in \Gal (\Q_q (N)/\Q_q(N/p))$
\begin{align*}
|\psi (\alpha)^q - \alpha^q|_p \leq p^{-1} \max (1, |\psi(\alpha)|_p)^q  \max (1, |\alpha|_p)^q .
\end{align*}
\end{lemma}

We want to replace the $p^{-1}$ in the lemma by something smaller.

\begin{lemma}
\label{adz}
Suppose $p|N$ and $\alpha \in \Q_q (N)^\times$. Then for all $\psi \in \Gal (\Q_q (N)/\Q_q(N/p))$ and $\lambda\in\N \setminus \{0\}$ we have
\begin{align*}
|(\psi (\alpha)^q)^{p^\lambda} - (\alpha^q)^{p^\lambda}|_p \leq p^{-s} \max (1, |\psi(\alpha)|_p)^{qt}  \max (1, |\alpha|_p)^{qt} 
\end{align*}
with $s = 1 + \lambda$ and $t = p^\lambda$.
\end{lemma}

The proof is essentially the same as in \cite{MR3182009}, Lemma 2.1.

\begin{proof}
Consider
\begin{align*}
|(\psi (\alpha)^q)^{p^\lambda} - (\alpha^q)^{p^\lambda}|_p = |\psi (\alpha)^q - \alpha^q|_p \prod_{j=1}^{\lambda} \prod_{\stackrel{\zeta \in \overline{\Q}_p}{\ord(\zeta) = p^j}} | \psi (\alpha)^q - \zeta \alpha^q|_p.
\end{align*}
By \cite{MR1697859}, Proposition 7.13 (i) (p. 159) and Theorem 4.8 (p. 131), we have $|1-\zeta|_p = p^{-\frac1{p^{j-1}(p-1)}}$ if $\ord \zeta = p^j$ and $j\geq 1$, so we get:
\begin{align*}
| \psi (\alpha)^q - \zeta \alpha^q|_p & = |\psi (\alpha)^q - \alpha^q + \alpha^q - \zeta \alpha^q|_p \\
& \leq \max(|\psi (\alpha)^q - \alpha^q|_p, |1-\zeta|_p |\alpha^q|_p)\\
& \stackrel{\text{Lemma }\ref{Lemma4.2}}{\leq} \max(p^{-1} \max (1, |\psi(\alpha)|_p)^q  \max (1, |\alpha|_p)^q , |1-\zeta|_p |\alpha^q|_p)\\
& = \max(p^{-1} \max (1, |\psi(\alpha)|_p)^q  \max (1, |\alpha|_p)^q , p^{-\frac1{p^{j-1}(p-1)}} |\alpha^q|_p)\\
& \leq p^{-\frac1{p^{j-1}(p-1)}} \max (1, |\psi(\alpha)|_p)^q  \max (1, |\alpha|_p)^q
\end{align*}
So we get
\begin{align*}
|(\psi (\alpha)^q)^{p^\lambda} - (\alpha^q)^{p^\lambda}|_p \leq p ^{-s} (\max (1, |\psi(\alpha)|_p)^q  \max (1, |\alpha|_p)^q)^t,
\end{align*}
where
\begin{align*}
s & = 1 + \sum_{j=1}^\lambda \sum_{\stackrel{\zeta \in \overline{\Q}_p}{\ord(\zeta) = p^j}} \frac1{p^{j-1}(p-1)} \\
& = 1 + \sum_{j=1}^\lambda p^{j-1}(p-1)\frac{1}{p^{j-1}(p-1)}\\
& = 1 + \lambda
\end{align*}
and 
\begin{align*}
t & = 1 + \sum_{j=1}^\lambda \sum_{\stackrel{\zeta \in \overline{\Q}_p}{\ord(\zeta) = p^j}} 1\\
& = 1 + \sum_{j=1}^\lambda p^{j-1}(p-1)\\
& = 1 + (p-1)\frac{p^\lambda -1}{p-1}\\
& = p^\lambda.
\end{align*}

\end{proof}

The next step is to reformulate Lemma 5.3 of \cite{MR3090783}. We try to get a similar result by using the above Lemmas.

\begin{lemma}
Let $\lambda \in \N$. We assume $p|N$ and let $n\geq 1$ be the greatest integer such that $p^n$ divides $N$. Let $Q(n) =\begin{cases}
 q &\mbox{if } n\geq 2,\\
 (q-1)q &\mbox{if } n=1\\
\end{cases}$. If $\alpha\in\Q(N)$ satisfies $\alpha^{Q(n)p^\lambda}\notin\Q_q (N/p)$, then
\begin{align}
h(\alpha) \geq \frac1{2{p^\lambda} p^2}\left(\frac{(1+\lambda) \log p}{p^6} - \log 2\right).
\end{align}
\end{lemma}

We follow the proof of Lemma 5.3 of \cite{MR3090783} very closely and use our Lemma \ref{adz} instead of Lemma 4.2 in \cite{MR3090783}.

\begin{proof}
For brevity, we set $Q=Q(n)$. By hypothesis we may choose $\psi\in \Gal(\Q_q(N)/\Q_q(N/p))$ with $\psi(\alpha^{Q {p^\lambda}}) \neq \alpha^{Q {p^\lambda}}$. We note that $\alpha \neq 0$.
We define
\begin{align*}
x = \psi(\alpha^{Q {p^\lambda}}) - \alpha^{Q {p^\lambda}} \in\Q(N)
\end{align*}
and observe $x\neq 0$ by our choice of $\psi$. So
\begin{align}
\label{productFormula}
\sum_w d_w \log |x|_w = 0
\end{align}
by the product formula.
Say $G = \{ \sigma \in \Gal (\Q (N) / \Q) | \sigma\psi\sigma^{-1} = \psi \}$ and $v$ is the place of $\Q (N)$ induced by $|\cdot|_p$. Let $\sigma\in G$. The place $\sigma v$ of $\Q(N)$ is defined by $|\sigma (y)|_{\sigma v} = |y|_v$ for all $y\in\Q(N)$. So $|(\sigma\psi\sigma^{-1})(\alpha^{Q})^{p^\lambda}-(\alpha^{Q})^{p^\lambda}|_{\sigma v} = |\psi(\sigma^{-1}(\alpha^{Q})^{p^\lambda})-\sigma^{-1}(\alpha^{Q})^{p^\lambda}|_v$. By definition we have $q|Q$, so we may apply Lemma \ref{adz} to $\sigma^{-1}(\alpha)^{\frac{Q}{q}}$. This implies
\begin{align*}
|(\sigma\psi\sigma^{-1})(\alpha^{Q{p^\lambda}})-\alpha^{Q{p^\lambda}}|_{\sigma v} & = |\psi(\sigma^{-1}(\alpha^Q))^{p^\lambda}-\sigma^{-1}(\alpha^Q)^{p^\lambda}|_{v}\\
& \leq p^{-(\lambda+1)} \max (1, |\psi(\sigma^{-1}(\alpha)^{\frac{Q}{q}})|_v)^{qp^\lambda}  \max (1, |\sigma^{-1}(\alpha)^{\frac{Q}{q}}|_v)^{qp^\lambda} \\
&\leq p^{-(\lambda+1)} \max (1, |(\sigma\psi\sigma^{-1})(\alpha)|_{\sigma v})^{Qp^\lambda}  \max (1, |\alpha|_{\sigma v})^{Qp^\lambda}.
\end{align*}
Now $\sigma \psi \sigma^{-1} = \psi$ since $\sigma\in G$. Therefore,
\begin{align}
\label{Gv}
|x|_w \leq p^{-(\lambda+1)} \max (1, |\psi(\alpha)|_w)^{Qp^\lambda}  \max (1, |\alpha|_w)^{Qp^\lambda} \text{ for all } w\in Gv.
\end{align}
If $w$ is an arbitrary finite place of $\Q(N)$, the ultrametric triangle inequality implies
\begin{align}
\label{finitePlaces}
|x|_w \leq \max(1,|\psi(\alpha)|_w)^{Qp^\lambda} \max(1, |\alpha|_w)^{Qp^\lambda}.
\end{align}
Say $w$ is an infinite place. Then the triangle inequality gives
\begin{align}
|x|_w &\leq |\psi(\alpha)|_w^{Qp^\lambda} + |\alpha|_w^{Q{p^\lambda}} \nonumber\\
& \leq 2\max(1,|\psi(\alpha)|_w)^{Qp^\lambda} \max(1, |\alpha|_w)^{Qp^\lambda} \label{infinitePlaces}.
\end{align}
We split the sum \eqref{productFormula} up into the finite places in $Gv$, the remaining finite places and the infinite places. The estimates \eqref{Gv}, \eqref{finitePlaces} and \eqref{infinitePlaces} together with the product formula \eqref{productFormula} imply
\begin{align*}
0 \leq& \sum_{w\in Gv} d_w \log (p^{-(\lambda+1)} \\
&+ \sum_{w \nmid \infty} d_w \log(\max(1,|\psi(\alpha)|_w)^{Qp^\lambda} \max(1, |\alpha|_w)^{Qp^\lambda})\\
&+ \sum_{w \mid \infty} d_w \log(2\max(1,|\psi(\alpha)|_w)^{Qp^\lambda} \max(1, |\alpha|_w)^{Qp^\lambda})\\
\leq& \sum_{w\in Gv} d_w \log p^{-(\lambda+1)} \\
&+ \sum_{w} d_w \log(\max(1,|\psi(\alpha)|_w)^{Qp^\lambda} \max(1, |\alpha|_w)^{Qp^\lambda})+ \sum_{w \mid \infty} d_w\log 2\\
\leq& -(\lambda+1)|Gv| d_v \log p + \sum_{w} d_w \log(\max(1,|\psi(\alpha)|_w)^{Qp^\lambda} \max(1, |\alpha|_w)^{Qp^\lambda})+ \sum_{w \mid \infty} d_w\log 2.
\end{align*}
Notice that all local degrees $d_w$ equal $d_v$ whenever $w \in Gv$. By Lemma 5.2 of \cite{MR3090783}, we have $|Gv| \geq \frac{[\Q(N):\Q]}{d_v p^4}$ hence we can divide the whole expression by $[\Q(N):\Q]$ and get
\begin{align*}
\frac{(\lambda+1) \log p}{p^4} &\leq Qp^\lambda (h(\psi(\alpha)) + h(\alpha)) + \log 2\\
&= 2 Qp^\lambda  h(\alpha) + \log 2.
\end{align*}
With $p^2 \leq Q\leq p^4$ and we get
\begin{align*}
h(\alpha) & \geq \frac{(\lambda+1) \log p}{2{p^\lambda} p^8} - \frac{\log 2}{2{p^\lambda} p^2} \\
& = \frac1{2{p^\lambda} p^2}\left(\frac{(\lambda+1) \log p}{p^6} - \log 2\right)\\
& = \frac1{2{p^\lambda} p^2}\left(\frac{(1+\lambda) \log p}{p^6} - \log 2\right).
\end{align*}
\end{proof}

\begin{corollary}
In the same setting as in the above lemma, but with $\lambda = p^6$ we have
\begin{align*}
h(\alpha) \geq \frac{\log \frac{p}{2}}{2p^{p^6+2}}
\end{align*}
\end{corollary}

\begin{proof}
We just continue where the above proof ended and set $\lambda = p^6$.\\

\begin{align}
h(\alpha) & \geq \frac1{2p^{p^6}p^2}\left(\frac{(1+p^6)\log p}{p^6} - \log 2\right) \\
& \geq \frac1{2p^{p^6+2}} (\log p - \log 2)\\
& \geq \frac{\log \frac{p}{2}}{2p^{p^6+2}}.
\end{align}
\end{proof}

This approach using $p$-adic amplification leads to worse dependency on $p$ when compared to the equidistribution approach in Section \ref{MignotteApproach}. So we stop here and leave this dead end.

\newpage

\section{Putting everything together to get an explicit lower bound }

We gathered all the results we need and are now able to prove the main theorems.
\begin{theorem}
\label{first}
Let $E$ be an elliptic curve defined over $\Q$ and let $p \geq 5$ be a supersingular prime of $E$ such that the Galois representation $\text{Gal}(\bar{\Q} / \Q) \to \text{Aut } E[p]$ is surjective. Then for $\alpha \in \Q(E_\text{tors})^\times \setminus \mu_\infty$ we have
\begin{align*}h(\alpha) \geq \frac{(\log p)^5}{10^{21} p^{44}}.\end{align*}
\end{theorem}

\begin{proof}
If $E$ has complex multiplication, then there exists a quadratic number field $K$ such that $K(E_\text{tors}) /K$ is abelian. Theorem 1.2 of \cite{MR2651944} tells us that the height of $\alpha\in K(E_\text{tors})^\times \setminus \mu_\infty$ is bounded from below by $3^{-14}$ which is always bigger than the bound in the Theorem.

So let us now look at the case where $E$ does not have complex multiplication. Proposition \ref{Philipp6.1} gives us $\beta \in \overline{\Q}^\times \setminus \mu_\infty$ of degree $d$ with $h(\beta) \leq 10p^4 h(\alpha)$ and
\begin{align*}h(\alpha) \geq \frac15 \left( \frac{\log p}{2p^8} - \max \Bigl\{ 0, \frac{1}{[\Q(\beta):\Q]}\sum_{\tau: \Q (\beta) \hookrightarrow \C} \log | \tau (\beta) - 1| \Bigr\} \right).\end{align*}
We want to distinguish two cases:

Case 1: $d\geq 10^{10} \geq e^{23}$ and $h(\alpha) \leq \frac1{40p^4}$.

Here we can use Lemma \ref{sumexpl} with $\delta=\frac3{10}$ and together with $h(\beta) \leq 10p^4 h(\alpha) \leq \frac14$ we get
\begin{align*}
h(\alpha) \geq \frac15 \left( \frac{\log p}{2p^8} - \frac{40}{(\frac3{10})^4}(10p^4 h(\alpha))^{1/5} \right)\\
 \geq \frac15 \left( \frac{\log p}{2p^8} - 4.94 \cdot 10^3(10p^4 h(\alpha))^{1/5} \right).
\end{align*}
Since $h(\alpha) \leq 1$ we can make use of the fact that $h(\alpha)^{1/5} \geq h(\alpha)$. Then we find that
\begin{align*}
h(\alpha)^{1/5} +\frac{4.94}{5}10^3 (10p^4 h(\alpha))^{1/5}  \geq \frac{\log p}{10p^8},
\end{align*}
which gives us
\begin{align*}
h(\alpha) \geq \left(\frac{1}{1+ \frac{4.94}{5} 10^3 (10p^4)^{1/5}} \frac{\log p}{10p^8}\right)^5.
\end{align*}
We can simplify this and get
\begin{align*}
h(\alpha) &\geq \left(\frac{1}{1+ \frac{4.94}{5} 10^3 (10p^4)^{1/5}} \frac{\log p}{10p^8}\right)^5 \\ &\geq \left(\frac{1}{10^3 (10p^4)^{1/5}} \frac{\log p}{10p^8}\right)^5 \\
& \geq \frac{(\log p)^5}{10^{21} p^{44}}.
\end{align*}
Case 2: $d\leq 10^{10}$.

In this case we easily get an estimate with the main theorem and Corollary 2 of \cite{MR1367580}:
\begin{align*}h(\beta) > \frac{1}{4\cdot10^{10}} \left( \frac{\log \log (10^{10})}{\log (10^{10})} \right)^3\end{align*}
which is greater than $6\cdot 10^{-14}$ hence $h(\alpha) \geq 6\cdot 10^{-14} \frac1{10p^4}$. This is always bigger than our bound from above so we proved the theorem.

\end{proof}

Since for a semistable elliptic curve the Galois representation is surjective for all $p\geq 11$ (Theorem 4 of \cite{MR482230}) we have the following corollary.

\begin{corollary}
Let $E$ be a semistable elliptic curve defined over $\Q$ and let $p \geq 11$ be a supersingular prime of $E$. Then for $\alpha \in \Q(E_\text{tors})^\times \setminus \mu_\infty$ we have
\begin{align*}h(\alpha) \geq \frac{(\log p)^5}{10^{21} p^{44}}.\end{align*}
\end{corollary}

Furthermore, we also know how to find a relatively small (which is actually pretty big) supersingular prime of $E$. Recall the definition of $B_E$ in Lemma \ref{Lemma1}. With $M = 11$, Theorem \ref{boundpnoj} gives us the next corollary.

\begin{corollary}
\label{semistable}
Let $E$ be a semistable elliptic curve defined over $\Q$ of conductor $N$. Let $\alpha \in \Q(E_\text{tors})^\times \setminus \mu_\infty$. Then with and $n =  \max (11, B_E)$ we have

\begin{align*}
h(\alpha) & \geq \frac{(\log 11)^5}{10^{21}} (2.5 \cdot 10^{9} e^{0.018 \sqrt{8 N e^{\vartheta (n)}} (\log ( 8 N e^{\vartheta (n)}))^3} N e^{\vartheta (n)} (\log ( 8 N e^{\vartheta (n)}))^6 N \log N)^{-44}.
\end{align*}
\end{corollary}

Theorem 4.2 of \cite{LeFourn2016} gives us a bound for surjective primes: A prime \newline $p \geq 10^7 \max\{985, \frac1{12}h(j_E)+3\}^2$ is always surjective.

\begin{remark}
In the next versions we do not have to care about the $B_E$ since \newline $10^7 \max\{985, \frac1{12}h (j_E)+3\}^2$ is always bigger than $B_E$.
\end{remark}

Here, we can get rid of the height of the $j$-invariant by bounding it via Theorem \ref{vK}: $10^7 \max\{985, \frac1{12} \max(\log 2,h(j_E)) +3\}^2 \leq 10^7 \max\{985, \frac1{12}(18N \log N)+3\}^2$.

\begin{theorem}
\label{heightbound}
Let $E$ be an elliptic curve defined over $\Q$ of conductor $N$. Let $\alpha \in \Q(E_\text{tors})^\times \setminus \mu_\infty$. Then with $n = 10^7 \max\{985, \frac1{12}(10 N \log N)+3\}^2$ we have

\begin{align*}
h(\alpha) \geq (2.5 \cdot 10^{9} e^{0.018 \sqrt{8 N e^{\vartheta (n)}} (\log ( 8 N e^{\vartheta (n)}))^3} N e^{\vartheta (n)} (\log ( 8 N e^{\vartheta (n)}))^6 \log N)^{-44}.
\end{align*}
\end{theorem}

If we are only interested in effective results, we can use Corollary \ref{effectiveElkies} and get the following effective, non-explicit result.

\begin{theorem}
Let $E$ be an elliptic curve defined over $\Q$ of conductor $N$ and $j$-invariant $j_E$. Let $\alpha \in \Q(E_\text{tors})^\times \setminus \mu_\infty$. Then with $q=4\rad (6N)$ and $n = 10^7 \max\{985, \frac1{12}h (j_E)+3\}^2$ there is an effectively computable constant $c \geq 0$ such that
\begin{align*}
h(\alpha) \geq c \frac{(\log (q (\log q) \max(\log 2,h(j_E))))^5}{(q^{5/2} (\log q)^2 \max(\log 2,h(j_E)))^{44}}.
\end{align*}
\end{theorem}

\section{Examples}
\label{examples}
In this section we want to give some examples of the height bound.

\begin{example}
Let $E: y^2 = x^3 + x$. Since $E$ has complex multiplication, we can refer to the proof of Theorem \ref{first}. Then for all $\alpha\in\Q(E_{\text{tor}})^\times\setminus \mu_\infty$ we have
\begin{align*}
h(\alpha) \geq 3^{-14}.
\end{align*}
\end{example}

For the next example we cite a result of Rosser and Sch\"onfeld.

\begin{theorem}[\cite{MR0137689}]
For $x > 0$ we have $\vartheta (x) < 1.01624 x$.
\end{theorem}

\begin{example}
Let $E: y^2 + xy = x^3-x$. This curve has the smallest possible conductor $11$. With $n= 10^7\cdot 985$ we can use Theorem \ref{heightbound}. We have that $8Ne^{\vartheta(n)} \geq 8 \cdot 11 \cdot e^{1.01624\cdot 10^7 \cdot 985} \geq e^{1.001\cdot 10^{10}}$. Then for all $\alpha\in\Q(E_{\text{tor}})^\times\setminus \mu_\infty$
\begin{align*}
h(\alpha) &\geq  (2.5 \cdot 10^{9} e^{0.018 e^{0.5.005\cdot 10^{10}} (1.001\cdot 10^{10})^3} e^{1.001\cdot 10^{10}} (1.001\cdot 10^{10})^6 \log 11)^{-44}\\
& \geq (e^{44\cdot (0.018 e^{0.5005\cdot 10^{10}} (1.001\cdot 10^{10})^3 + 1.001\cdot 10^{10} + \log (2.5 \cdot 10^{9}  (1.001\cdot 10^{10})^6 \log 11))})^{-1}\\
& \geq (e^{0.792\cdot e^{0.5005\cdot 10^{10}}\cdot 1.01\cdot 10^{30} + 1.001 \cdot 10^{10} + 10^3})^{-1}\\
& \geq (e^{10^{30}\cdot e^{0.5005\cdot 10^{10}}})^{-1}\\
& \geq(e^{e^{10^{10}}})^{-1}.
\end{align*}
\end{example}

\begin{example}
Let $E: y^2 + y = x^3 - x^2$. Then by \cite{MR903384} the prime $19$ is supersingular for $E$ and by \cite[\href{http://www.lmfdb.org/EllipticCurve/Q/280105/a/1}{Elliptic Curve 280105.a1}]{lmfdb}, all primes but $5$ are surjective. So for all $\alpha\in\Q(E_{\text{tor}})^\times\setminus \mu_\infty$ we have
\begin{align*}
h(\alpha) &\geq \frac{(\log 19)^5}{10^{21} 19^{44}}\\
&\geq 10^{-66}.
\end{align*}
\end{example}

\bibliography{Bibliographie}
\bibliographystyle{alpha}
\parindent0pt
\end{document}